\documentclass[12pt]{article}
\title{Indecomposable objects determined by their index in Higher Homological Algebra}
\author{Joseph Reid}
\date{}
\usepackage{fancyhdr}
\usepackage{eqparbox}
\usepackage{amsfonts}
\usepackage{amsthm}
\usepackage{amsmath}
\usepackage{amssymb}
\usepackage{wasysym}
\usepackage{mathrsfs}
\usepackage[margin=1in]{geometry}
\theoremstyle{definition}

\newtheorem{theorem}{Theorem}[section]
\newtheorem*{theorem*}{Theorem}
\newtheorem{corollary}[theorem]{Corollary}
\newtheorem{proposition}[theorem]{Proposition}
\newtheorem{lemma}[theorem]{Lemma}
\newtheorem{remark}[theorem]{Remark}

\newtheorem{definition}[theorem]{Definition}

\usepackage{tikz}
\usetikzlibrary{arrows,positioning}
\usetikzlibrary{decorations.markings}

\tikzset{
  big dot/.style={
    circle, inner sep=0pt, 
    minimum size=3mm, fill=black
 }
}
\tikzset{
  normal dot/.style={
    circle, inner sep=0pt, 
    minimum size=1.5mm, fill=black
 }
}

\newcommand{\homs}{\textrm{Hom}_{\mathscr{C}}}
\newcommand{\tilthoms}{\textrm{Hom}_{\frac{\mathscr{C}}{[\Sigma^d \mathscr{T}]}}}

\begin{document}
\maketitle
\thispagestyle{fancy}
ABSTRACT. Let $\mathscr{C}$ be a 2-Calabi-Yau triangulated category, and let $\mathscr{T}$ be a cluster tilting subcategory of $\mathscr{C}$. An important result from Dehy and Keller tells us that a rigid object $c \in \mathscr{C}$ is uniquely defined by its index with respect to $\mathscr{T}$.

The notion of triangulated categories extends to the notion of $(d+2)$-angulated categories. Thanks to a paper by Oppermann and Thomas, we now have a definition for cluster tilting subcategories in higher dimensions. This paper proves that under a technical assumption, an indecomposable object in a $(d+2)$-angulated category is uniquely defined by its index with respect to a higher dimensional cluster tilting subcategory. We also demonstrate that this result applies to higher dimensional cluster categories of Dynkin type $A$.
\section{Introduction}
When dealing with triangulated categories, the index is an interesting invariant. In order to define the index we must first make the following definition, which is originally due to Iyama in the abelian case, see \cite[Definition~2.2]{Iyama}:
\begin{definition}
Let $\mathscr{C}$ be a triangulated category with translation functor $\Sigma$. Then a full subcategory $\mathscr{T}$ is a \textit{cluster tilting subcategory} if:
\begin{itemize}
\item[(i)] $\mathscr{T}$ is contravariantly and covariantly finite.
\item[(ii)] $X \in \mathscr{T}$ if and only if $\textrm{Ext}(X,\mathscr{T}) = 0$.
\item[(iii)] $X \in \mathscr{T}$ if and only if $\textrm{Ext}(\mathscr{T},X) = 0$.
\end{itemize}

If there exists an object $T \in \mathscr{T}$ such that $\textrm{add}(T) = \mathscr{T}$, then we call $T$ a \textit{cluster tilting object}.

\end{definition} 

We have the result from \cite[Lemma~3.2(1)]{KoenigZhu} that given a triangulated category $\mathscr{C}$ with a cluster tilting subcategory $\mathscr{T}$, for every $c \in \mathscr{C}$, there exists a triangle
\begin{align*}
t_1 \to t_0 \to c \to \Sigma(t_1)
\end{align*}
in $\mathscr{C}$ such that $t_0, t_1 \in \mathscr{T}$.

If we have a cluster tilting subcategory $\mathscr{T}$ we may construct the \textit{split Grothendieck group} for $\mathscr{T}$, which we denote $K_0^{\textrm{Split}}(\mathscr{T})$. This group is the abelian group generated by the objects of $\mathscr{T}$, modulo all the relations of the form $[t]=[t_0] + [t_1]$ when $t \cong t_0 \oplus t_1$. Using this, we may define the notion of index:

\begin{definition}[\cite{Palu}]\label{indDef}
Let $\mathscr{C}$ be a triangulated category, and let $\mathscr{T}$ be a cluster tilting subcategory. We define $\textrm{Ind}_{\mathscr{T}}: \mathscr{C} \to K_0^{\textrm{Split}}(\mathscr{T})$. For $c \in \mathscr{C}$ there is a triangle 
\begin{align*}
t_1 \to t_0 \to c \to \Sigma(t_1)
\end{align*}
such that $t_0, t_1 \in \mathscr{T}$. Then
\begin{align*}
\textrm{Ind}_{\mathscr{T}}(c) = [t_0] - [t_1].
\end{align*}
This is well defined by \cite[Lemma~2.1]{Palu}.
\end{definition}

Use of the index has given us some powerful tools in the study of triangulated categories. One key result which has inspired this paper is:

\begin{theorem*}[Dehy-Keller, {\cite[Theorem~2.3]{DehyKeller}}]
Let $K$ be an algebraically closed field, let $\mathscr{C}$ be a $K$-linear Hom-finite triangulated category with split idempotents, and assume also that $\mathscr{C}$ is 2-Calabi-Yau. Let $\mathscr{T}$ be a cluster tilting subcategory of $\mathscr{C}$. Then the index as defined in definition \ref{indDef} induces an injection from the set of isomorphism classes of rigid objects into $K_0^{\textrm{Split}}(\mathscr{T})$.
\end{theorem*}

Another source of inspiration is the following similar result from classic representation theory; note that it concerns indecomposable objects only.

\begin{theorem*}[Auslander-Reiten-Smal$\o$, {\cite[Theorem~IX.4.7]{ARS}}]
Let $\Lambda$ be an Artin algebra over an algebraically closed field $K$, and let $M$ and $N$ be indecomposable finite dimensional $\Lambda$-modules. Let $P_1(M) \to P_0(M) \to M \to 0$ and $P_1(N) \to P_0(N) \to N \to 0$ be the minimal projective resolutions of $M$ and $N$ respectively. If $P_1(M) \cong P_1(N)$ and $P_0(M) \cong P_0(N)$, and $M$ and $N$ are not the start of so-called short chains, then $M \cong N$.
\end{theorem*}

The concept of triangulated categories extends to the notion of $(d+2)$-angulated categories, and the notion of cluster tilting subcategories extends to the notion of Oppermann-Thomas cluster tilting subcategories, see \cite[Definition~2.1]{GKO} and \cite[Definition~5.3]{OppermannThomas}. Importantly, we may also define the index in a $(d+2)$-angulated category with respect to an Oppermann-Thomas cluster tilting subcategory, see \cite[Definition~B]{JustJorg}. We also here fix some notation: Let $\mathscr{C}$ be a category, and $\mathscr{T}$ some subcategory of $\mathscr{C}$. Then for any $x, y \in \mathscr{C}$ we denote by $\homs{}^{[\mathscr{T}]}(x,y)$ the space made up of elements of $\homs{}(x,y)$ that factor through some $t \in \mathscr{T}$

Using these definitions, this paper will prove the following results:
\begingroup
\renewcommand{\thetheorem}{\Alph{theorem}}
\setcounter{theorem}{0}
\begin{theorem}\label{theRmk}
Let $\mathscr{C}$ be a $K$-linear, Hom-finite $(d+2)$-angulated category with split idempotents and $d$ odd, and let $\mathscr{T} =\textrm{add}(T)$ be an Oppermann-Thomas cluster tilting subcategory. Assume that if $c, x \in \mathscr{C}$ are indecomposable, then $\tilthoms{}(c, x)$ and $\homs{}^{[\mathscr{T}]}(\Sigma^{-d}(c),x)$ cannot be simultaneously non-zero. Then each indecomposable object $c \in \mathscr{C}$ is uniquely determined by its index with respect to $\mathscr{T}$ up to isomorphism.
\end{theorem}

\begin{corollary}\label{corB}
Let $\mathscr{C}$ be a $K$-linear, Hom-finite, $2d$-Calabi-Yau $(d+2)$-angulated category with split idempotents and $d$ odd, and let $\mathscr{T} =\textrm{add}(T)$ be an Oppermann-Thomas cluster tilting subcategory. Assume that if $c, x \in \mathscr{C}$ are indecomposable, then $\tilthoms{}(c, x)$ and $\tilthoms{}(x, \Sigma^d(c))$ cannot be simultaneously non-zero. Then each indecomposable object $c \in \mathscr{C}$ is uniquely determined by its index with respect to $\mathscr{T}$ up to isomorphism.
\end{corollary}

As with triangulated categories, thanks again to Oppermann and Thomas, there is a natural extension of cluster categories into higher dimensions. We will prove the following:
\begin{theorem}\label{thmB}
Let $\mathscr{C} = \mathscr{C}(A_n^d)$ be the $(d+2)$-angulated Oppermann-Thomas cluster category of Dynkin type $A_n$ with $d$ odd, and let $\mathscr{T}$ be an Oppermann-Thomas cluster tilting subcategory, see \cite[Section~6]{OppermannThomas}. Then each indecomposable object $c \in \mathscr{C}$ is uniquely determined by its index with respect to $\mathscr{T}$ up to isomorphism.
\end{theorem}

\endgroup
\section{Definitions}
We begin with some definitions. For the purpose of this paper, $K$ is an algebraically closed field.
\begin{definition}\cite[Definition~2.1]{GKO}
Let $\mathscr{C}$ be an additive category with an automorphism denoted $\Sigma^d$, where $d$ is a fixed positive integer. The inverse is denoted $\Sigma^{-d}$, but we note that $\Sigma^d$ is not assumed to be the $d$-th power of another functor. Then a \textit{$\Sigma^d$-sequence} in $\mathscr{C}$ is a diagram of the form
\begin{align}\label{theEqn}
c^0 \xrightarrow{\gamma^0} c^1 \rightarrow c^2 \rightarrow \cdots \rightarrow c^d \rightarrow c^{d+1} \xrightarrow{\gamma^{d+1}} \Sigma^d(c^0).
\end{align}
\end{definition}

\begin{definition}[{\cite[Definition~2.1]{GKO}}]
A \textit{$(d+2)$-angulated category} is a triple ($\mathscr{C}, \Sigma^d, \pentagon$) where $\pentagon$ is a class of $\Sigma^d$-sequences called $(d+2)$-angles, satisfying the following conditions:
\begin{itemize}
\item [(N1)] $\pentagon$ is closed under sums and summands, and contains the $(d+2)$-angle 
\begin{align*}
c \xrightarrow{id_c} c \rightarrow 0 \rightarrow \cdots \rightarrow 0 \rightarrow 0 \rightarrow \Sigma^d(c)
\end{align*}
for each $c \in \mathscr{C}$. For each morphism $c^0 \xrightarrow{\gamma^0} c^1$ in $\mathscr{C}$, the class $\pentagon$ contains a $\Sigma^d$-sequence of the form in Equation (\ref{theEqn}).
\item [(N2)] The $\Sigma^d$-sequence (1) is in $\pentagon$ if and only if the $\Sigma^d$-sequence
\begin{align*}
c^1 \xrightarrow{\gamma^1} c^2 \xrightarrow{\gamma^2} c^3 \rightarrow \cdots \rightarrow c^{d+1} \xrightarrow{\gamma^{d+1}} \Sigma^d(c^0) \xrightarrow{(-1)^d\Sigma^d(\gamma^0)} \Sigma^d(c^1)
\end{align*}
is in $\pentagon$. This sequence is known as the left rotation of sequence (1).
\item [(N3)] A commutative diagram with rows in $\pentagon$ has the following extension property shown with dotted arrows: \\
\begin{tikzpicture}[line cap = round, line join = round]

\node (a) at (-3, 2) {$b^0$};
\node (b) at (-1, 2) {$b^1$};
\node (c) at (1, 2) {$b^2$};
\node (d) at (3, 2) {$\ldots$};
\node (e) at (5, 2) {$b^{d}$};
\node (f) at (7, 2) {$b^{d+1}$};
\node (g) at (9, 2) {$\Sigma(b^0)$};
\node (h) at (-3, 0) {$c^0$};
\node (i) at (-1, 0) {$c^1$};
\node (j) at (1, 0) {$c^2$};
\node (k) at (3, 0) {$\ldots$};
\node (l) at (5, 0) {$c^{d}$};
\node (m) at (7, 0) {$c^{d+1}$};
\node (n) at (9, 0) {$\Sigma(c^0)$};

%Top row
\draw [->] (a) to (b);
\draw [->] (b) to (c);
\draw [->] (c) to (d);
\draw [->] (d) to (e);
\draw [->] (e) to (f);
\draw [->] (f) to (g);

%Bottom row
\draw [->] (h) to (i);
\draw [->] (i) to (j);
\draw [->] (j) to (k);
\draw [->] (k) to (l);
\draw [->] (l) to (m);
\draw [->] (m) to (n);

%Between
\draw [->] (a) to node[pos=0.5, right] {$\beta_0$} (h);
\draw [->] (b) to (i);
\draw [dotted, ->] (c) to (j);
\draw [dotted, ->] (e) to (l);
\draw [dotted, ->] (f) to (m);
\draw [->] (g) to node[pos=0.5, right] {$\Sigma(\beta_0)$} (n);

\end{tikzpicture}
\item[(N4)] The Octahedral Axiom, see \cite[Definition~2.1]{GKO}.
\end{itemize}
\end{definition}

\begin{definition}[{\cite[Definition~5.3]{OppermannThomas}}]
Let ($\mathscr{C}, \Sigma^d, \pentagon$) be a $(d+2)$-angulated category, and let $T \in \mathscr{T}$ such that $\mathscr{T} = \textrm{add}(T)$ is a full subcategory of $\mathscr{C}$. We call $T$ an \textit{Oppermann-Thomas cluster tilting object} of $\mathscr{C}$ if:
\begin{itemize}
\item[(i)] $\homs{}(\mathscr{T}, \Sigma^d(\mathscr{T})) = 0$,
\item[(ii)] for any $c \in \mathscr{C}$, there exists a $(d+2)$-angle
\begin{align}
t_d \rightarrow t_{d-1} \rightarrow \cdots \rightarrow t_1 \rightarrow t_0 \rightarrow c \rightarrow \Sigma^d(t_d) \label{theseq}
\end{align}
where $t_i \in \mathscr{T}$ for each $i$.
\end{itemize}
In this case, $\mathscr{T} = \textrm{add}(T)$ is an \textit{Oppermann-Thomas cluster tilting subcategory}.
\end{definition}

\begin{definition}[{\cite[Definitions~2.2, 2.4]{Jasso}}]
Let $\mathscr{F}$ be an additive category.
\begin{itemize}
\item[(i)] A diagram $f_{d+1} \rightarrow \cdots \rightarrow f_2 \rightarrow f_1$ is a \textit{$d$-kernel} of a morphism $f_1 \to f_0$ if the sequence
\begin{align*}
0 \to f_{d+1} \rightarrow \cdots \rightarrow f_1 \rightarrow f_0
\end{align*}
becomes exact under the functor $\textrm{Hom}_{\mathscr{F}}(f', -)$ for each $f' \in \mathscr{F}$.
\item[(ii)] A diagram $f_{d} \rightarrow f_{d-1} \rightarrow \cdots \rightarrow f_1 \rightarrow f_0$ is a \textit{$d$-cokernel} of a morphism $f_{d+1} \to f_d$ if the sequence
\begin{align*}
f_{d+1} \rightarrow \cdots \rightarrow f_1 \rightarrow f_0 \to 0
\end{align*}
becomes exact under the functor $\textrm{Hom}_{\mathscr{F}}(-, f')$ for each $f' \in \mathscr{F}$.
\item[(iii)] A \textit{$d$-exact sequence} is a diagram
\begin{align*}
f_{d+1} \to f_d \to f_{d-1} \to \cdots \to f_2 \to f_1 \to f_0
\end{align*}
such that $f_{d+1} \to f_d \to f_{d-1} \to \cdots \to f_2 \to f_1$ is a $d$-kernel of $f_1 \to f_0$ and $f_d \to f_{d-1} \to \cdots \to f_2 \to f_1 \to f_0$ is a $d$-cokernel of $f_{d+1} \to f_d$. We will often write these sequences in the form
\begin{align*}
0 \to f_{d+1} \to f_d \to f_{d-1} \to \cdots \to f_2 \to f_1 \to f_0 \to 0.
\end{align*}
\end{itemize}
\end{definition}

\begin{definition}[{\cite[Definition~3.1]{Jasso}}]
An additive category $\mathscr{F}$ is called \textit{$d$-abelian} if it has the following properties:
\begin{itemize}
\item[(A0)] $\mathscr{F}$ has split idempotents.
\item[(A1)] Each morphism in $\mathscr{F}$ has a $d$-kernel and a $d$-cokernel.
\item[(A2)] If $f_{d+1} \to f_d$ is a monomorphism with $d$-cokernel $f_d \to f_{d-1} \to \cdots \to f_2 \to f_1 \to f_0$, then
\begin{align*}
0 \to f_{d+1} \to f_d \to f_{d-1} \to \cdots \to f_2 \to f_1 \to f_0 \to 0
\end{align*}
is a $d$-exact sequence.
\item[(A2')] The dual of (A2).
\end{itemize}
\end{definition}

\begin{definition}[{\cite[Definition~2.2]{Iyama}, \cite[Definition~3.14]{Jasso}}]
Let $\mathscr{C}$ be an abelian or a triangulated category. A full subcategory $\mathscr{F}$ of $\mathscr{C}$ is a \textit{$d$-cluster tilting subcategory} if
\begin{itemize}
\item[(i)]\eqmakebox[things][l]{$\mathscr{F}$}
     $ \begin{aligned}[t]
      &= \{c \in \mathscr{C}|\textrm{Ext}_{\mathscr{C}}^{1, 2, \ldots, d-1}(\mathscr{F}, c) = 0 \} \\
&= \{c \in \mathscr{C}|\textrm{Ext}_{\mathscr{C}}^{1, 2, \ldots, d-1}(c, \mathscr{F}) = 0 \}.
      \end{aligned} $
\item[(ii)] $\mathscr{F}$ is functorially finite.
\item[(iii)] If $\mathscr{C}$ is abelian then $\mathscr{F}$ is generating and cogenerating.
\end{itemize}
\end{definition}

\begin{definition}\label{resDef}
Let $\mathscr{C}$ be an abelian category, and $\mathscr{F}$ a full subcategory. An \textit{augmented left $\mathscr{F}$-resolution} of $c \in \mathscr{C}$ is a sequence
\begin{align*}
\ldots \to f_2 \to f_1 \to f_0 \to c \to 0
\end{align*}
with $f_i \in \mathscr{F}$ for each $i$, which becomes exact under the functor $\homs{}(f, -)$ for every $f \in \mathscr{F}$. In this case, the sequence
\begin{align*}
\ldots \to f_2 \to f_1 \to f_0  \to 0 \to \ldots
\end{align*}
is called a \textit{left $\mathscr{F}$-resolution} of $c$. The right resolutions are defined dually.
\end{definition}

\begin{definition}
Let $\mathscr{C}$ be a $(d+2)$-angulated category, and let $D=\textrm{Hom}_K(-, K)$ be the usual duality functor. A \textit{Serre functor} for $\mathscr{C}$ is an auto-equivalence $S: \mathscr{C} \to \mathscr{C}$ together with a family of isomorphisms which are natural in $X$ and $Y$:
\begin{align*}
t_{X, Y}:\homs{}(Y, SX) \to D\homs{}(X, Y).
\end{align*}
We call the category $\mathscr{C}$ \textit{$2d$-Calabi-Yau} if $\mathscr{C}$ admits a Serre functor which is isomorphic to $(\Sigma^d)^2$, which we often write $\Sigma^{2d}$.
\end{definition}

The following was proven in a special case by Oppermann and Thomas \cite[Theorem~5.6]{OppermannThomas} and more generally by Jacobsen and J{$\o$}rgensen as part (i) of \cite[Theorem~0.5]{JorgJac}:
\begin{theorem}\label{JorgJac}
Suppose that ($\mathscr{C}, \Sigma^d, \pentagon$) is a $K$-linear Hom-finite $(d+2)$-angulated category with split idempotents, and that $T$ is an Oppermann-Thomas cluster tilting object such that $\mathscr{T} = \textrm{add}(T)$. Then the quotient $\frac{\mathscr{C}}{[\Sigma^d\mathscr{T}]}$ can be identified with a $d$-cluster tilting subcategory $\mathscr{D}$ of mod $\Lambda$, where $\Lambda=\textrm{End}(T)$. In particular, the quotient is $d$-abelian.
\end{theorem}
\section{Proof of Theorem A}
\begin{proposition} \label{theprop}
Let $\mathscr{C}$ be a $K$-linear, Hom-finite $(d+2)$-angulated category, and let $\mathscr{T}$ be an Oppermann-Thomas cluster tilting subcategory. Suppose that for $c \in \mathscr{C}$ we have a $(d+2)$-angle
\begin{align*}
t_d \rightarrow t_{d-1} \rightarrow \cdots \rightarrow t_1 \rightarrow t_0 \rightarrow c \rightarrow \Sigma^d(t_d)
\end{align*}
with $t_i \in \mathscr{T}$. Then for any $x \in \mathscr{C}$ the following are true:
\begin{itemize}
\item[(i)] There is an exact sequence
\begin{align*}
0 \to& \tilthoms{}(c, x) \to \homs{}(t_0, x) \to \homs{}(t_1, x) \to \\
& \cdots \to \homs{}(t_d, x) \to \homs{}^{[\mathscr{T}]}(\Sigma^{-d}(c),x) \to 0.
\end{align*}
\item[(ii)]\eqmakebox[things2][2]{}
     $\begin{aligned}[t]
      \textrm{dim}_K \tilthoms{}(c, x) + (-1)^d \textrm{dim}_K \homs{}^{[\Sigma^d \mathscr{T}]}(c,\Sigma^d(x))= \Sigma_{i=0}^{i=d} (-1)^i \textrm{dim}_K \homs{}(t_i, x).
      \end{aligned} $
%\begin{align*}
%\textrm{dim}_K \tilthoms{}(c, x) + (-1)^d \textrm{dim}_K \tilthoms{}(x, \Sigma^d(c))= \Sigma_{i=0}^{i=d} (-1)^i \textrm{dim}_K \homs{}(t_i, x).
%\end{align*}
\end{itemize}
\end{proposition}
Before the proof we show a lemma:
\begin{lemma}\label{theLemma}
Suppose we have $\mathscr{C, T}$ as in Proposition \ref{theprop}, and we have the same $(d+2)$-angle. There is an exact sequence
\begin{align*}
0 \to& \tilthoms{}(c, x) \to \homs{}(t_0, x) \to \homs{}(t_1, x) \to \cdots \to \homs{}(t_d, x).
\end{align*}
\end{lemma}
\begin{proof}
We start with the $(d+2)$-angle as stated in Proposition \ref{theprop}. Due to the second axiom of $(d+2)$-angulated categories, we have a longer sequence
\begin{align*}
\cdots \to \Sigma^{-d}(t_0) \to \Sigma^{-d}(c) \to t_d \rightarrow t_{d-1} \rightarrow \cdots& \rightarrow t_1 \rightarrow t_0 \rightarrow c \rightarrow \Sigma^d(t_d) \rightarrow \Sigma^d(t_{d-1}) \to \cdots&
\end{align*}
which becomes exact under the action of the functor $\homs{}(-, x)$ for $x \in \mathscr{C}$ by \cite[Proposition~2.5(a)]{GKO}. This gives us a long exact sequence
\begin{align*}
\cdots \to& \homs{}(\Sigma^d(t_{d-1}), x) \to \homs{}(\Sigma^d(t_d), x) \to \homs{}(c, x) \to \\
\cdots \to& \homs{}(t_{d-1}, x) \to \homs{}(t_d, x) \to \homs{}(\Sigma^{-d}(c), x) \to \cdots.
\end{align*}
In particular, there is an exact sequence
\begin{align*}
\homs{}(\Sigma^d(t_d), x) \to \homs{}(c, x) \to 
\cdots \to \homs{}(t_{d-1}, x) \to \homs{}(t_d, x).
\end{align*}

We examine this sequence more closely. By labelling the maps in the $(d+2)$-angle:
\begin{align*}
\ldots \to t_1 \xrightarrow{\delta_1} t_0 \xrightarrow{\delta_0} c \xrightarrow{\phi} \Sigma^d(t_d) \to \ldots,
\end{align*}
we obtain labels for our sequence:
\begin{align*}
\ldots \to \homs{}(\Sigma^d(t_d), x) \xrightarrow{\phi^*} \homs{}(c, x) \xrightarrow{\delta_0^*} \homs{}(t_0, x) \xrightarrow{\delta_1^*} \homs{}(t_1, x) \to \ldots.
\end{align*}

To establish the lemma, it is enough to prove that $\textrm{Coker}(\phi^*) \cong \tilthoms{}(c, x)$. This is equivalent to $\textrm{Ker}(\delta_0^*)$ being equal to $\textrm{Hom}_{\mathscr{C}}^{[\Sigma^d\mathscr{T}]}(c, x)$, which is what we will aim to show.

By exactness, $\textrm{Ker}(\delta_0^*) = \textrm{Im}(\phi^*)$, so we have immediately that $\textrm{Ker}(\delta_0^*) \subseteq \textrm{Hom}_{\mathscr{C}}^{[\Sigma^d\mathscr{T}]}(c, x)$. Take any $\mu \in \textrm{Hom}_{\mathscr{C}}^{[\Sigma^d\mathscr{T}]}(c, x)$. Then as $\textrm{Hom}(\mathscr{T}, \Sigma^d \mathscr{T}) = 0$, we know that $\mu \circ \delta_0 = 0$. This means that $\mu \in \textrm{Ker}(\delta_0^*)$. This gives us that $\textrm{Hom}_{\mathscr{C}}^{[\Sigma^d\mathscr{T}]}(c, x) \subseteq \textrm{Ker}(\delta_0^*)$, and so we have proven the equality.\\
\end{proof}
Since $\homs{}(-, x)$ and $\tilthoms{}(-,x)$ are equivalent on $\mathscr{T}$, we see that this result can actually be restated in the following way:
\begin{remark}
Suppose we have $\mathscr{C, T}$ as in Proposition \ref{theprop}, and we have the same $(d+2)$-angle. Then in $\frac{\mathscr{C}}{[\Sigma^d \mathscr{T}]}$, the sequence
\begin{align*}
t_{d-1} \to t_{d-2} \to \cdots \to t_0 \to c
\end{align*}
is a $d$-cokernel for the morphism $t_d \to t_{d-1}$.
\end{remark}

We now prove Proposition \ref{theprop}.
\begin{proof}
Part (i): \\
We have the exact sequence from Lemma \ref{theLemma}. It remains to prove that there is also an exact sequence 
\begin{align*}
&\homs{}(t_0, x) \to \homs{}(t_1, x) \to \cdots \to \homs{}(t_d, x) \to \homs{}^{[\mathscr{T}]}(\Sigma^{-d}(c),x) \to 0.
\end{align*}
As in the proof of the lemma, we have the long exact sequence
\begin{align*}
\cdots \to& \homs{}(\Sigma^d(t_{d-1}), x) \to \homs{}(\Sigma^d(t_d), x) \to \homs{}(c, x) \to \\
\cdots \to& \homs{}(t_d, x) \to \homs{}(\Sigma^{-d}(c), x) \to \homs{}(\Sigma^{-d}(t_0), x) \to \cdots.
\end{align*}
In particular, there is an exact sequence
\begin{align*}
\homs{}(t_0, x)  \xrightarrow{\delta_1^*} \cdots \xrightarrow{\delta_{d}^*}  \homs{}(t_d, x) \xrightarrow{(\Sigma^{-d} \phi)^*}  \homs{}(\Sigma^{-d}(c), x) \xrightarrow{(\Sigma^{-d} \delta_0)^*} \homs{}(\Sigma^{-d}(t_0), x).
\end{align*}
We have used labels similar to those in the proof of Lemma \ref{theLemma}.
To establish the proposition, it is enough to prove that $\textrm{Ker}((\Sigma^{-d} \delta_0)^*) = \homs{}^{[\mathscr{T}]}(\Sigma^{-d}(c),x)$.\\

Take $\theta \in \homs{}(\Sigma^{-d}(c), x)$. Firstly, suppose that $\theta \in \homs{}^{[\mathscr{T}]}(\Sigma^{-d}(c),x)$. Then $(\Sigma^{-d} \delta_0)^*(\theta) = \theta \circ \Sigma^{-d}\delta_0 = 0$, or we would have a non-zero morphism from $\Sigma^{-d}\mathscr{T}$ to $\mathscr{T}$ which contradicts the definition of $\mathscr{T}$. So we have that $\homs{}^{[\mathscr{T}]}(\Sigma^{-d}(c),x) \subseteq \textrm{Ker}((\Sigma^{-d} \delta_0)^*)$.\\

Now suppose that $\theta \in \textrm{Ker}((\Sigma^{-d} \delta_0)^*)$. By exactness, this means that $\theta \in \textrm{Im}((\Sigma^{-d}\phi)^*)$. That is, $\theta = \theta' \circ \Sigma^{-d}\phi$ for some $\theta': t_d \to x$. Then obviously $\theta \in \homs{}^{[\mathscr{T}]}(\Sigma^{-d}(c),x)$. This gives us that $\textrm{Ker}((\Sigma^{-d} \delta_0)^*) \subseteq \homs{}^{[\mathscr{T}]}(\Sigma^{-d}(c),x)$. Then we have the required equality, which proves part(i).

Part (ii):
We now have an exact sequence
\begin{align*}
0 \to& \tilthoms{}(c, x) \to \homs{}(t_0, x) \to \homs{}(t_1, x) \to \\
& \cdots \to \homs{}(t_d, x) \to \homs{}^{[\mathscr{T}]}(\Sigma^{-d}(c),x) \to 0.
\end{align*}

This exact sequence gives us the equation in part (ii) by the Rank-Nullity theorem, given that $\textrm{dim}_K\homs{}^{[\Sigma^d\mathscr{T}]}(c,\Sigma^d(x)) = \textrm{dim}_K \homs{}^{[\mathscr{T}]}(\Sigma^{-d}(c),x)$.
\end{proof}

We can extract more information from this set of results in the case that $\mathscr{C}$ is a $2d$-Calabi-Yau category.
\begin{proposition}\label{restateProp}
Let $\mathscr{C}$ be a $K$-linear, Hom-finite, $2d$-Calabi-Yau $(d+2)$-angulated category, and let $\mathscr{T}$ be an Oppermann-Thomas cluster tilting subcategory. Then there are isomorphisms
\begin{align*}
\homs{}^{[\Sigma^d \mathscr{T}]}(c,\Sigma^d(x)) \cong D\tilthoms{}(x, \Sigma^d(c))
\end{align*}
for any two objects $c, x \in \mathscr{C}$.
\end{proposition}
\begin{proof}
By definition, we have a $(d+2)$-angle
\begin{align*}
t_d \to \ldots \to t_0 \xrightarrow{\phi} c \xrightarrow{\delta} \Sigma^d (t_d)
\end{align*}
with each $t_i \in \mathscr{T}$. We can use this to give us a long exact sequence
\begin{align}\label{aseq}
\ldots \to \homs{}(\Sigma^d(t_d),\Sigma^d(x)) \xrightarrow{\delta^*} \homs{}(c,\Sigma^d(x)) \xrightarrow{\phi^*} \homs{}(t_0, \Sigma^d(x)) \to \ldots.
\end{align}
We will show first that $\textrm{Im}(\delta^*) = \homs{}^{[\Sigma^d \mathscr{T}]}(c,\Sigma^d(x))$. \\

It is clear that $\textrm{Im}(\delta^*) \subseteq \homs{}^{[\Sigma^d \mathscr{T}]}(c,\Sigma^d(x))$ as all such maps must factor through $\Sigma^d(t_d)$. Now, take $\theta \in \homs{}^{[\Sigma^d \mathscr{T}]}(c,\Sigma^d(x))$. Then $\phi^*(\theta) = 0$, or we would have a non-zero map from $\mathscr{T}$ to $\Sigma^d \mathscr{T}$. So $\theta \in \textrm{Ker}(\phi^*)$, and by exactness $\theta \in \textrm{Im}(\delta^*)$. This gives the desired equality. \\

We look again at the long exact sequence (\ref{aseq}). We dualise the sequence to obtain the long exact sequence
\begin{align*}
\ldots \to D\homs{}(t_0, \Sigma^d(x)) \xrightarrow{D\phi^*} D\homs{}(c,\Sigma^d(x)) \xrightarrow{D\delta^*} D\homs{}(\Sigma^d(t_d),\Sigma^d(x)) \to \ldots.
\end{align*}
We may apply the Serre duality to obtain the sequence
\begin{align*}
\ldots \to \homs{}(x, \Sigma^d(t_0)) \xrightarrow{(\Sigma^d \phi)_*} \homs{}(x, \Sigma^d(c)) \xrightarrow{(\Sigma^d \delta)_*} \homs{}(x, \Sigma^{2d}(t_d)) \to \ldots
\end{align*}
where due to the Serre duality being natural, we know that $\textrm{Im}((\Sigma^d \delta)_*) \cong D \textrm{Im}(\delta^*)$. \\

We wish to examine $\textrm{Im}((\Sigma^d \delta)_*)$. By exactness, $\textrm{Im}((\Sigma^d \delta)_*) \cong \frac{\homs{}(x, \Sigma^d(c))}{\textrm{Im}((\Sigma^d \phi)_*)}.$ We claim that $\textrm{Im}((\Sigma^d \phi)_*) = \homs{}^{[\Sigma^d \mathscr{T}]}(x, \Sigma^d(c))$.

The inclusion $\textrm{Im}((\Sigma^d \phi)_*) \subseteq \homs{}^{[\Sigma^d \mathscr{T}]}(x, \Sigma^d(c))$ is clear. To see the other inclusion, take $\theta \in \homs{}^{[\Sigma^d \mathscr{T}]}(x, \Sigma^d(c))$. Then $(\Sigma^d \delta)_*(\theta) = 0$ or we would have a non-zero map from $\Sigma^d \mathscr{T}$ to $\Sigma^{2d} \mathscr{T}$ which is a contradiction. Then by exactness, $\theta \in \textrm{Im}((\Sigma^d \phi)_*)$ and we have proven the equality. This means that $\textrm{Im}((\Sigma^d \delta)_*) \cong \tilthoms{}(x, \Sigma^d(c))$, and as such that $D \textrm{Im}(\delta^*) \cong \tilthoms{}(x, \Sigma^d(c))$.
Then we have that
\begin{align*}
\homs{}^{[\Sigma^d \mathscr{T}]}(c,\Sigma^d(x)) &= \textrm{Im}(\delta^*) \\
&\cong D\tilthoms{}(x, \Sigma^d(c))
\end{align*}
as required.
\end{proof}

\begin{corollary}\label{thepropCY}
Let $\mathscr{C}$ be a $K$-linear, Hom-finite, $2d$-Calabi-Yau $(d+2)$-angulated category, and let $\mathscr{T}$ be an Oppermann-Thomas cluster tilting subcategory. Suppose that for $c \in \mathscr{C}$ we have a $(d+2)$-angle
\begin{align*}
t_d \rightarrow t_{d-1} \rightarrow \cdots \rightarrow t_1 \rightarrow t_0 \rightarrow c \rightarrow \Sigma^d(t_d)
\end{align*}
with $t_i \in \mathscr{T}$. Then for any $x \in \mathscr{C}$
\begin{align*}
\textrm{dim}_K \tilthoms{}(c, x) + (-1)^d \textrm{dim}_K \tilthoms{}(x,\Sigma^d(c))= \Sigma_{i=0}^{i=d} (-1)^i \textrm{dim}_K \homs{}(t_i, x).
\end{align*}
\end{corollary}

We have already defined the index in a triangulated category in definition \ref{indDef}. We may extend this definition to $(d+2)$-angulated categories. We see that if we have an Oppermann-Thomas cluster tilting category $\mathscr{T}$, we can construct the split Grothendieck group $K_0^{\textrm{Split}}(\mathscr{T})$ in the same way as in the triangulated case. Then we define the index as the following:

\begin{definition}[{\cite[Definition~B]{JustJorg}}]\label{higherIndDef}
Let $\mathscr{C}$ be a $K$-linear Hom-finite $(d+2)$-angulated category with split idempotents. The \textit{index} of an object $c \in \mathscr{C}$ with respect to an Oppermann-Thomas cluster tilting subcategory $\mathscr{T}$ is defined as:
\begin{align*}
\textrm{Ind}_{\mathscr{T}}(c) = \Sigma_{i=0}^d (-1)^i[t_i]
\end{align*}
where
\begin{align*}
t_d \rightarrow t_{d-1} \rightarrow \cdots \rightarrow t_1 \rightarrow t_0 \rightarrow c \rightarrow \Sigma^d(t_d)
\end{align*}
is a $(d+2)$-angle with each $t_i \in \mathscr{T}$. It follows from \cite[Remark~5.4]{JustJorg} that the index is well defined.
\end{definition}

We show another two propositions, which will allow us to prove Theorem \ref{theRmk}.
\begin{proposition}\label{extraProp1}
Let $\mathscr{C}$ be a $K$-linear, Hom-finite $(d+2)$-angulated category where $d$ is odd, and let $\mathscr{T}$ be an Oppermann-Thomas cluster tilting subcategory. Assume that if $c, x \in \mathscr{C}$ are indecomposable, then $\tilthoms{}(c, x)$ and $\homs{}^{[\Sigma^d\mathscr{T}]}(c,\Sigma^d(x))$ cannot be simultaneously non-zero. Then
\begin{align*}
\textrm{dim}_K \tilthoms{}(c, x) = \textrm{max}(0, \Sigma_{i=0}^{i=d} (-1)^i \textrm{dim}_K \homs{}(t_i, x))
\end{align*}
and
\begin{align*}
-\textrm{dim}_K \homs{}^{[\Sigma^d\mathscr{T}]}(c,\Sigma^d(x)) = \textrm{min}(0, \Sigma_{i=0}^{i=d} (-1)^i \textrm{dim}_K \homs{}(t_i, x)).
\end{align*}
In particular, these values are determined by $\textrm{Ind}_{\mathscr{T}}(c)$ and $x$.
\end{proposition}
\begin{proof}
By Proposition \ref{theprop} part (ii), we already know that 
\begin{align*}
\textrm{dim}_K \tilthoms{}(c, x) + (-1)^d \textrm{dim}_K \homs{}^{[\Sigma^d\mathscr{T}]}(c,\Sigma^d(x))= \Sigma_{i=0}^{i=d} (-1)^i \textrm{dim}_K \homs{}(t_i, x).
\end{align*}
By the assumptions in the proposition, this means that when $\Sigma_{i=0}^{i=d} (-1)^i \textrm{dim}_K \homs{}(t_i, x) >0$ we have that 
\begin{align*}
\textrm{dim}_K \tilthoms{}(c, x)= \Sigma_{i=0}^{i=d} (-1)^i \textrm{dim}_K \homs{}(t_i, x).
\end{align*}
Similarly, when $\Sigma_{i=0}^{i=d} (-1)^i \textrm{dim}_K \homs{}(t_i, x) < 0$ we have that 
\begin{align*}
- \textrm{dim}_K \homs{}^{[\Sigma^d\mathscr{T}]}(c,\Sigma^d(x))= \Sigma_{i=0}^{i=d} (-1)^i \textrm{dim}_K \homs{}(t_i, x).
\end{align*}
Finally, if $\Sigma_{i=0}^{i=d} (-1)^i \textrm{dim}_K \homs{}(t_i, x) = 0$ then 
\begin{align*}
\textrm{dim}_K \tilthoms{}(c, x) = \textrm{dim}_K \homs{}^{[\Sigma^d\mathscr{T}]}(c,\Sigma^d(x))= 0.
\end{align*}
This is equivalent to the claim made in the statement of this proposition. The fact that these values are determined by $\textrm{Ind}_{\mathscr{T}}(c)$ and $x$ follows from the definitions.
\end{proof}

We state a remark from Iyama \cite{Iyama} that will allow us to finish the proof of Theorem \ref{theRmk}.

\begin{remark}[{\cite[Proposition~2.2.2]{Iyama}}]\label{resRmk}
Let $\mathscr{A}$ be an abelian category, and let $\mathscr{F}$ be a $d$-cluster tilting subcategory of $\mathscr{A}$. Then each $a \in \mathscr{A}$ has left and right $\mathscr{F}$-resolutions as in Definition \ref{resDef} of length $\leq d-1$.
\end{remark}

We show a final proposition to prove Theorem \ref{theRmk}.

\begin{proposition}\label{extraProp2}
Let $\mathscr{D}$ be a $d$-cluster tilting subcategory of $\textrm{mod}$ $\Lambda$, where $\Lambda$ is a finite-dimensional $K$-algebra. If $c \in \mathscr{D}$ and $\textrm{dim}_K\textrm{Hom}_{\mathscr{D}}(c, x)$ is known for each indecomposable $x \in \mathscr{D}$, then $c$ is determined up to isomorphism.
\end{proposition}

\begin{proof}
By Remark \ref{resRmk}, each $m \in \textrm{mod}$ $\Lambda$ has an augmented $\mathscr{D}$-left resolution
\begin{align*}
0 \to x_{d-1} \to x_{d-2} \to \ldots \to x_0 \to m \to 0
\end{align*}
which gives an exact sequence
\begin{align*}
0 \to \textrm{Hom}_{\Lambda}(c, x_{d-1}) \to \ldots \to \textrm{Hom}_{\Lambda}(c, x_0) \to \textrm{Hom}_{\Lambda}(c, m) \to 0
\end{align*}
for all $c \in \mathscr{D}$ by Definition \ref{resDef}. By assumption, we know $\textrm{dim}_K\textrm{Hom}_{\Lambda}(c, x_i)$ for each $x_i$. By the Rank-Nullity theorem, we know $\textrm{dim}_K\textrm{Hom}_{\Lambda}(c, m)$. By Auslander \cite[Corollary~1.2]{AusReit}, we then have determined $c$ up to isomorphism. 
\end{proof}

\paragraph{Proof of Theorem A}
\begin{proof}
By Proposition \ref{extraProp1} the index $\textrm{Ind}_{\mathscr{T}}(c)$ determines $\textrm{dim}_K \tilthoms{}(c,x)$ for each indecomposable $x \in \mathscr{C}$. If $\textrm{dim}_K \tilthoms{}(c,x)$ is zero for each indecomposable $x \in \mathscr{C}$, then $c$ is zero in $\frac{\mathscr{C}}{[\Sigma^d\mathscr{T}]}$ so $c \in \Sigma^d\mathscr{T}$. Hence there is a $(d+2)$-angle
\begin{align*}
t_d \to \ldots \to t_0 \to c \to \Sigma^d t_d
\end{align*}
with $t_d =\Sigma^{-d} c$ and all other $t_i$ equal to zero. It follows that $\textrm{Ind}_{\mathscr{T}}(c) = (-1)^d[\Sigma^{-d} c] = -[\Sigma^{-d} c]$, and this determines $c$ up to isomorphism. \\

If $\textrm{dim}_K \tilthoms{}(c,x)$ is not zero for each indecomposable $x \in \mathscr{C}$, then $c$ is non-zero in $\frac{\mathscr{C}}{[\Sigma^d\mathscr{T}]}$ so $c$ is not in $\Sigma^d\mathscr{T}$. We can use Theorem \ref{JorgJac} to identify $\frac{\mathscr{C}}{[\Sigma^d\mathscr{T}]}$ with $\mathscr{D}$, a $d$-cluster tilting subcategory of $\textrm{mod}$ $\Lambda$ where $\Lambda = \textrm{End}(T)$. Then $\textrm{Ind}_{\mathscr{T}}(c)$ determines $\textrm{dim}_K \tilthoms{}(c,x)$ for each indecomposable $x \in \mathscr{D}$, so proposition \ref{extraProp2} means that $\textrm{Ind}_{\mathscr{T}}(c)$ determines $c$ up to isomorphism in $\mathscr{D}$, hence in $\frac{\mathscr{C}}{[\Sigma^d\mathscr{T}]}$. As the projection functor $\mathscr{C} \to \frac{\mathscr{C}}{[\Sigma^d \mathscr{T}]}$ gives a bijection from the isoclasses of indecomposable objects outside of $\Sigma^d \mathscr{T}$ to isoclasses of all indecomposables, we have that $\textrm{Ind}_{\mathscr{T}}(c)$ determines $c$ up to isomorphism in $\mathscr{C}$ because $c$ is not in $\Sigma^d\mathscr{T}$.
\end{proof}

We note here that the proof of Corollary \ref{corB} follows directly from Theorem \ref{theRmk} and Proposition \ref{restateProp}.

\section{Proof of Theorem C}
We will now demonstrate a class of categories where we may apply Theorem \ref{theRmk}. We use the class of higher dimensional cluster categories defined in \cite[Sections~5 and 6]{OppermannThomas}, which are analogous to the cluster categories of type $A_n$. We will label them as $\mathscr{C}(A_n^d)$, which is the $(d+2)$-angulated \textit{Oppermann-Thomas cluster category} of Dynkin type $A_n$. \\
The following description of $\mathscr{C}(A_n^d)$ is a restatement of Propositions 3.12 and 6.1 and Lemma 6.6(2) in \cite{OppermannThomas}. We take the canonical cyclic ordering of the set $V = \{ 1, 2, \ldots, n + 2d+1\}$, which it can be helpful to think of as the vertices of an $(n + 2d + 1)$-gon labelled in a clockwise direction. This means that for three points in our ordering $x, y, z$ such that $x < y < z$, if we start at $x$ and move clockwise, we will encounter first $y$ then $z$. It is worth noting that if we have $x < y < z$, then we also have that $y < z \leq x$ and $z \leq x <  y$. For a point $x$ in our ordering, we denote by $x^-$ the vertex of our polygon that is one step anticlockwise of $x$.\\

\begin{proposition}\label{propBij}
The indecomposable objects of $\mathscr{C}(A_n^d)$ are in bijection with subsets of $V$ that have size $d+1$ and contain no neighbouring vertices. We identify each indecomposable $X$ with its subset of $V$, and will write $X = \{x_0, x_1, \ldots, x_d\}$.
\end{proposition}
We see immediately that by setting $d=1$ in proposition \ref{propBij}, we obtain the indecomposables of the traditional cluster category of type $A_n$ \cite[Section~2.2]{CCS}.

Using the identification described in proposition \ref{propBij}, we can easily describe the action of the translation functor, and also how the indecomposable objects interact with one another.
\begin{proposition}\label{propFunc}
The translation functor simply shifts an indecomposable by one place; that is, if $X = \{x_0, x_1, \ldots, x_d\}$, then $\Sigma^d(X) = \{x_0^-, x_1^-, \ldots, x_d^-\}$.
\end{proposition}

\begin{definition}\label{intertwine}
For two indecomposable objects $X$ and $Y$ of $\mathscr{C}(A_n^d)$, we say that $X$ and $Y$ \textit{intertwine} if there is a labelling of $X = \{x_0, x_1, \ldots, x_d\}$ and of $Y=\{y_0, y_1, \ldots, y_d\}$ such that 
\begin{align*}
x_0 < y_0 < x_1 < y_1 < x_2 < \ldots < x_d < y_d < x_0.
\end{align*}
\end{definition}
We see that definition \ref{intertwine} is symmetric; we take $Y = Y'$, where we choose the labelling as $y_i' = y_{i-1}$ for $1 \leq i \leq d$ and $y_0' = y_d$. This gives us that
\begin{align*}
y_0' < x_0 < y_1' < x_1 < y_2' <\ldots < y_d' < x_d < y_0'
\end{align*}
as required. \\

For two indecomposable objects $X$ and $Y$ of $\mathscr{C}(A_n^d)$, either $\textrm{Hom}(X, Y) = 0$ or $\textrm{Hom}(X, Y) = K$; this is the same as in the classic cluster category. In fact, we have the following: 
\begin{proposition}[{\cite[Proposition~6.1]{OppermannThomas}}]\label{homProp}
For two indecomposable objects $X$ and $Y$ of $\mathscr{C}(A_n^d)$, we have $\textrm{Hom}(X, Y) = K$ if and only if $X$ and $\Sigma^{-d}(Y)$ intertwine. This is equivalent to $X$ and $Y$ having labellings such that the following is true:
\begin{align*}
x_0 \leq y_0 \leq x_1^{--} < x_1 \leq y_1 \leq x_2^{--} < \ldots < x_d \leq y_d \leq x_0^{--}.
\end{align*}
\end{proposition}

We may also speak to whether or not there is a factorisation of a non-zero homomorphism in $\mathscr{C}(A_n^d)$.
\begin{proposition}[{\cite[Proposition~3.12]{OppermannThomas}}]\label{propFact}
For two indecomposable objects $X$ and $Y$ of $\mathscr{C}(A_n^d)$ satisfying the condition in Proposition \ref{homProp}, a non-zero morphism $X \to Y$ factors through a third irreducible $Z$ if and only if there exists a labelling for $Z=\{z_0, z_1, \ldots, z_d\}$ such that
\begin{align*}
x_0 \leq z_0 \leq y_0, x_1 \leq z_1 \leq y_1, \ldots, x_d \leq z_d \leq y_d.
\end{align*}
\end{proposition}

It is also true, again due to \cite{OppermannThomas}, that our categories $\mathscr{C}(A_n^d)$ permit Oppermann-Thomas cluster tilting objects. By \cite[Theorem~2.4]{OppermannThomas} and \cite[Theorem~6.4]{OppermannThomas}, the sum $T$ of ${n + d - 1 \choose d}$ mutually non-intertwining indecomposable objects is an Oppermann-Thomas cluster tilting object. Combining this with \cite[Lemma~6.6]{OppermannThomas} gives us the following proposition:
\begin{proposition}
The sum $T$ of ${n + d - 1 \choose d}$ mutually non-intertwining indecomposable objects of $\mathscr{C}(A_n^d)$ is an Oppermann-Thomas cluster tilting object. Moreover, this describes all such objects. These objects are maximal with respect to the non-intertwining property.
\end{proposition}

We claim that our categories $\mathscr{C}(A_n^d)$ serve as an example for the use of Theorem \ref{theRmk}, thereby proving Theorem \ref{thmB} which follows from combining the following Lemma with Theorem \ref{theRmk}.
\begin{lemma}\label{lastLemma}
Let $\mathscr{C} = \mathscr{C}(A_n^d)$ be the $(d+2)$-angulated Oppermann-Thomas cluster category of Dynkin type $A_n$, and let $\mathscr{T} = \textrm{add}(T)$ be an Oppermann-Thomas cluster tilting subcategory. If $c, x \in \mathscr{C}$ are indecomposable, then $\tilthoms{}(c, x)$ and $\tilthoms{}(x, \Sigma^d(c))$ cannot be simultaneously non-zero.
\end{lemma}
\begin{proof}
We begin with the assumption that $\tilthoms{}(c, x) \neq 0$ and $\tilthoms{}(x, \Sigma^d(c)) \neq 0$. It follows that $\homs{}(c, x) \neq 0$ so by Proposition \ref{homProp} there exist labellings $c = \{c_0, c_1, \ldots, c_d\}$ and $x = \{x_0, x_1, \ldots, x_d\}$ such that
\begin{align}\label{importantseq}
c_0 \leq x_0 \leq c_1^{--} < c_1 \leq x_1 \leq c_2^{--} < \ldots < c_d \leq x_d \leq c_0^{--}.
\end{align}

Similarly, we know that $\homs{}(x, \Sigma^d(c)) \neq 0$ and by applying the Serre duality, we see that $\homs{}(c, \Sigma^d(x)) \neq 0$. Combining this with (\ref{importantseq}) and Propositions \ref{propFunc} and \ref{homProp}, we see that the labellings of $c$ and $x$ are such that
\begin{align}\label{anotherImportantSeq}
c_0 \leq x_0^- \leq c_1^{--} < c_1 \leq x_1^- \leq c_2^{--} < \ldots < c_d \leq x_d^- \leq c_0^{--}.
\end{align} \\

By \cite[Lemma~6.1]{JustJorg}, $\tilthoms{}(x, \Sigma^d(c)) \neq 0$ implies that $\tilthoms{}(c, \Sigma^d(x)) = 0$. From Proposition \ref{propFact} and (\ref{anotherImportantSeq}), there exists an $s = \{s_0, s_1, \ldots, s_d\} \in \Sigma^d \mathscr{T}$ such that
\begin{align*}
c_0 \leq s_0 \leq x_0^- , c_1 \leq s_1 \leq x_1^-, \ldots, c_d \leq s_d \leq x_d^-.
\end{align*}
Combining this with (\ref{importantseq}) shows that
\begin{align*}
c_0 \leq s_0 \leq x_0 , c_1 \leq s_1 \leq x_1, \ldots, c_d \leq s_d \leq x_d
\end{align*}
which by Proposition \ref{propFact} means that $\tilthoms{}(c, x) = 0$. This contradicts our initial assumption, and proves the lemma.
\end{proof}
\paragraph{Proof of Theorem C}
\begin{proof}
$\mathscr{C} = \mathscr{C}(A_n^d)$ is a $K$-linear Hom-finite $2d$-Calabi-Yau $(d+2)$-angulated category with split idempotents. By Lemma \ref{lastLemma} for $c, x \in \mathscr{C}$ indecomposable, $\tilthoms{}(c, x)$ and $\tilthoms{}(x, \Sigma^d(c))$ cannot be simultaneously non-zero. Thus by Theorem A, when $d$ is odd, we have the result.
\end{proof}
It is natural to ask whether the assumption that $d$ is odd is necessary, as there is no analogue in the triangulated case. We provide a remark to show that it is:
\begin{remark}
Theorems A and B do not generalise to the case of $d$ being even.
\end{remark}
\begin{proof}
Suppose that $\mathscr{C}$ and $\mathscr{T}$ are as in the statement of Theorem B, except assume that $d=2$. Take some indecomposable $t \in \mathscr{T}$. Then $\textrm{Ind}_{\mathscr{T}}(t) = [t]$. We also have, by the definition of a $(d+2)$-angulated category, the trivial $4$-angle
\begin{align*}
t \to t \to 0 \to 0 \to \Sigma^d t,
\end{align*}
which gives the $4$-angle
\begin{align*}
t \to 0 \to 0 \to \Sigma^d t \to \Sigma^d t.
\end{align*}
Notice that this $4$-angle is of the form given in definition \ref{higherIndDef}, and shows that
\begin{align*}
\textrm{Ind}_{\mathscr{T}}(\Sigma^d t) = (-1)^2[t] = [t].
\end{align*}
Thus, if $t$ and $\Sigma^d t$  are not isomorphic, then we have two indecomposable objects of $\mathscr{T}$ that have equal index with respect to $\mathscr{T}$ and as such cannot be uniquely defined by their index. \\

An example of this is in the category $\mathscr{C}(A_2^2)$. We obtain an Oppermann-Thomas cluster tilting subcategory of $\mathscr{C}(A_2^2)$ by selecting all of the indecomposable elements containing the vertex $1$. We will call this subcategory $\mathscr{U}$. This includes the indecomposable $u = (1, 3, 5)$. Then $\Sigma^d u = (2, 4, 7)$. These two indecomposables are clearly not isomorphic, but they have the same index.
\end{proof}
\bibliography{mybib}{}
\bibliographystyle{mybib}
School of Mathematics and Statistics, Newcastle University, Newcastle upon Tyne, NE1 7RU, United Kingdom \\
\textit{Email address}: j.reid4@ncl.ac.uk
\end{document}